\documentclass[twoside,11pt,a4paper]{amsart}
\usepackage{amssymb,supertabular,a4}
\usepackage{amsmath,amsthm,amscd,amssymb,graphicx}
\usepackage{blkarray}
\input xy
\xyoption{all}
\def\Hom{\hbox{\rm Hom\,}}
\def\Ext{\hbox{\rm Ext\,}}

\def\dimHom{\hbox{\rm dim\,\Hom\,}}
\def\dimExt{\hbox{\rm dim\,\Ext\,}}

\def\ql{\hbox{\rm ql\,}}
\def\qs{\hbox{\rm qs\,}}
\def\qt{\hbox{\rm qt\,}}

\def\udim{\hbox{\rm \underline{dim}\,}}
\def\ra{{\,\rightarrow\,}}

\def\sup{\hbox{\rm sup}}

\title[]
{Dimension vectors in regular components over wild Kronecker quivers}
\author{Bo Chen}
\email {mcebbchen@googlemail.com}

\address{Institut f\"ur Algebra und Zahlentheorie\\
                          Universit\"at Stuttgart\\
                           Pfaffenwaldring 57\\
                           D-70569, Stuttgart\\ Germany
                           }


\thanks{}

\newtheorem{theo}{Theorem}[section]

\newtheorem{lemm}[theo]{Lemma}

\newtheorem{prop}[theo]{Proposition}

\newtheorem{ques}[theo]{Question}

\newcommand{\matindex}[1]{\mbox{\scriptsize#1}}

\newenvironment{example}[1][Example]{\begin{trivlist}
\item[\hskip \labelsep {\bfseries #1}]}{\end{trivlist}}
\newenvironment{remark}[1][Remark]{\begin{trivlist}
\item[\hskip \labelsep {\bfseries #1}]}{\end{trivlist}}

\begin{document}

\begin{abstract} Let $\mathcal{K}_n$  be the so-called  wild Kronecker quiver, i.e., a quiver
with one source and one sink and $n\geq 3$ arrows from the source to the sink. The following problems  will be studied for an arbitrary regular component $\mathcal{C}$ of the Auslander-Reiten quiver: (1) What is the relationship between dimension vectors and quasi-lengths of the indecomposable regular representations in $\mathcal{C}$?  (2) For a given natural number $d$, is there an upper bound of the number of indecomposable representations in $\mathcal{C}$ with the same length $d$? (3) When do the sets of the dimension vectors of indecomposable representations in different regular components coincide? 
\end{abstract}

\maketitle


\section{Introduction}

Quivers (without oriented cycles) and their representations over a field play an important role in representation theory of finite-dimensional algebras; they also occur in other domains including Kac-Moody Lie algebras, quantum groups.  Gabriel characterized  quivers of finite representation type, that is, having only finitely many isomorphism classes of indecomposable representation: such quivers are exactly the disjoint union of Dynkin diagrams of types $A_n$, $D_n$, $E_6$, $E_7$, $E_8$, equipped with arbitrary orientations. Moreover, the isomorphism classes of indecomposable representations correspond bijectively to the positive roots of the associated root systems.  Whereas, the underlying graphs of quivers of tame representation type were shown to be the disjoint union of Euclidean diagrams of types $\widetilde{A}_n$, $\widetilde{D}_n$, $\widetilde{E}_6$, $\widetilde{E}_7$, $\widetilde{E}_8$.  Any quiver, which is neither of finite nor of tame type, is of wild representation type.   Kac generalized Gabriel's result to arbitrary quivers (exposed in \cite{KR}):  the dimension vectors of indecomposable representations are positive roots, and moreover, for each positive real root there is a uniquely isomorphism class of indecomposable representation, whereas for each positive imaginary root there are infinitely many ones.  

Representations of quivers of  finite and tame types are well understood. However, almost all quivers are of wild representation type and their representations are very complicated.  The so-called Auslander-Reiten (\emph{AR}) quiver of a connected wild quiver consisting a preprojective component including all indecomposable projective representations, a preinjective component including indecomposable injective representations, and infinitely many regular components of the form $A^\infty_\infty$: 
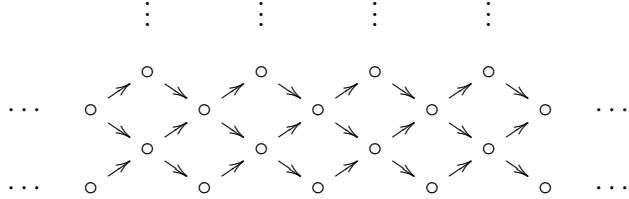
\[\xymatrix@R=2pt@C=8pt{
  &&\vdots && \vdots && \vdots && \vdots &&\\
  &&\circ\ar[rd] && \circ\ar[rd] &&
         \circ\ar[rd] && \circ\ar[rd] &&\\
  \cdots&\circ\ar[ru]\ar[rd] && \circ\ar[ru]\ar[rd] && \circ\ar[ru]\ar[rd] &&
         \circ\ar[ru]\ar[rd] && \circ & \cdots\\
  &&\circ\ar[ru]\ar[rd] && \circ\ar[ru]\ar[rd] &&
         \circ\ar[ru]\ar[rd] && \circ\ar[ru]\ar[rd] &&\\
   \cdots&\circ\ar[ru] && \circ\ar[ru] && \circ \ar[ru] && \circ\ar[ru] && \circ & \cdots\\
}\]
Kac's theorem tells the existence of indecomposable representations for every positive root. But it doesn't help to determine the positions of the indecomposable representations in the components (more interesting, regular components) of the AR quiver.

Given a wild quiver $Q$, a regular component $\mathcal{C}$  of its AR quiver and a natural number $d>0$, let $\beta_{Q, \mathcal{C}}(d)$ denote the number of the indecomposable representations 
in $\mathcal{C}$ with length $d$, and let $\beta_{Q,d}=\sup_{\mathcal{C}}\beta_{Q, \mathcal{C}}(d)$ and  $\beta_{Q}=\sup_{d}\beta_{Q, d}$, where the supremes are taken over all regular components and all natural numbers $d$, respectively.  In \cite{Z}, it was proved (for any Artin wild hereditary algebra) that
the indecomposable representations in a regular component of the AR quiver are uniquely determined by their dimension vectors . It follows immediately that $\beta_{Q,\mathcal{C}}(d)$ is finite for any regular component $\mathcal{C}$ and $d>0$ and $\beta_{Q,d}$ is obvious finite since $Q$ has only finitely many vertices. However, it is not known yet  if $\beta_Q$ is finite.

In this paper, we will focus on a class of quivers and their representations over an algebraically closed field $k$: the so-called Kronecker quivers 
$\mathcal{K}_n$ with $n$ arrows
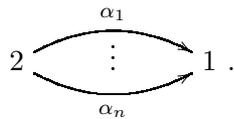
\[\xymatrix{ 2\ar@/^1pc/[rr]^{\alpha_1}\ar@/_1pc/[rr]_{\alpha_n}&\vdots& 1\\}.\]
A finite-dimensional representation of $\mathcal{K}_n$ over $k$ is simply called a $\mathcal{K}_n$-module.
It is well-known that $\mathcal{K}_n$ is of wild (resp. tame,  finite) representation type  if $n\geq 3$ (resp. $n=2$, $n=1$).
From now on,  unless stated otherwise, we always consider wild Kronecker quivers, i.e., $n\geq 3$. A representation of $\mathcal{K}_n$ over $k$ is simply called a $\mathcal{K}_n$-module. 

One of our main aims of the paper is to study the relationship between the dimension vectors of indecomposable $\mathcal{K}_n$-modules and their quasi-lengths, that is, on which layers of the regular components they locate.  For each $r\geq 0$, let $A_r$ be the $r$-th number appeared in the ordered set of the entries of the dimensional vectors of the indecomposable preprojective $\mathcal{K}_n$-modules (see Section 2).  We will see that if $(a,b)$ is an imaginary root,  then there is an indecomposable module $M$ with dimension vector $(a,b)$ and quasi-length $r$ if and only if $A_r$ is a common divisor of $a$ and $b$ (Theorem \ref{ql}). This proof is based on the construction of bricks (quasi-simple modules) for any positive imaginary root (Theorem \ref{bricks}).

A second aim of the paper is to study the indecomposable modules in a regular components having the same length.  
It will be shown for a given regular component $\mathcal{C}$ and a natural number $d$ that  $\beta_{\mathcal{K}_n}(\mathcal{C},d)\leq 2$ (Theorem \ref{bound}) and thus $\beta_{\mathcal{K}_n}\leq 2$. 
The equality holds almost everywhere in the following sense:  
for any
pair of natural numbers $r$ and $s$, there is always a regular component containing
indecomposable modules $M$ and $N$ with quasi-lengths $r$ and $s$, respectively, such that they have the same length: 
$|M|=|N|$ (Theorem \ref{samelength}).
A more interesting case is that 
two indecomposable $\mathcal{K}_n$-modules have the same length  and locate in the same $\tau$-orbit, where $\tau$ denotes the AR translation. In this case, we are able to describe the dimension vectors of the indecomposable modules (Proposition \ref{sameorbit}) like what has been done in \cite{C} for $3$-Kronecker quivers using Fibonacci numbers. 

The last aim of the paper is to study the coincidence of the sets of dimension vectors of indecomposable modules in different regular components. Many regular components may have the same set of dimension vectors. For a regular component $\mathcal{C}$, let $\udim \mathcal{C}$ the set of dimension vectors of the indecomposable modules in $\mathcal{C}$.   We will see that $\udim \mathcal{C}$ is uniquely determined by a dimension vector of an indecomposable module and its quasi-length, and also uniquely determined by two dimension vectors of indecomposable modules  not being in the same Coxeter orbits (Theorem \ref{twoorbits}).

\section{Positive roots and  real roots}

Throughout, we assume $k$ is an algebraically closed field, $\mathcal{K}_n$ is a wild Kronecker quiver with $n\geq 3$ arrows.  The representations of $\mathcal{K}_n$ over $k$ are simply called modules.  
We refer to \cite{ARS, R}
for general facts of representations of quivers and to \cite{K,KR} for the connection between the dimension vectors of indecomposable modules, and the roots of the corresponding Kac-Moody Lie algebras.

\subsection{Roots and their Coxeter orbits} The quadratic form associated to the wild Kronecker quiver $\mathcal{K}_n$ is
$q((x_1,x_2))=x_1^2+x_2^2-nx_1x_2$. A vector $(a,b)\in\mathbb{Z}^2$ is called a real root
if $q((a,b))=1$ and is called an imaginary root if $q((a,b))<0$, i.e. $\frac{n-\sqrt{n^2-4}}{2}<\frac{a}{b}<\frac{n+\sqrt{n^2-4}}{2}$.  
The dimension vector of an indecomposable module is either a positive
real root or a positive imaginary root. Conversely,
for each positive real root $(a,b)$, there are precisely one isomorphism class
of indecomposable modules with dimension
vector $(a,b)$.  These indecomposable modules are either preprojective or
 preinjective. Whereas for each positive imaginary
root $(a,b)$ there are infinitely many
indecomposable modules.
The Euler form is $\langle
(x_1,x_2),(y_1,y_2) \rangle=x_1y_1+x_2y_2-nx_1y_2$. For two
indecomposable modules $X$ and $Y$,
$\dimHom(X,Y)-\dimExt^1(X,Y)=\langle\udim X,\udim Y\rangle.$
The
Cartan matrix and the Coxeter matrix are the following:
 $$  C=\left(\begin{array}{cc}
                     1 & 0 \\
                     n & 1 \\
             \end{array}\right),
             \quad
     \Phi=-C^{-t}C=
       \left(\begin{array}{cr}
                        n^2-1 & n \\
                        -n & -1   \\
             \end{array}\right),
             \quad
     \Phi^{-1}=
       \left(\begin{array}{rc}
                        -1 & -n   \\
                        n & n^2-1 \\
             \end{array}\right).
 $$
The dimension vectors of the indecomposable modules
can be calculated using $\udim\tau M=(\udim
M)\Phi$ if $M$ is not projective and $\udim \tau^{-1}N=(\udim
N)\Phi^{-1}$ if $N$ is not injective, where $\tau$ denotes the
AR translation. A Coxeter orbit of an indecomposable module  $M$ 
consists of the indecomposable module $\tau^i M$ for all $i\in\mathbb{Z}$

\subsection{AR components} The AR-quiver consists of a preprojective component,
a preinjective
component and infinitely many regular ones.
The preprojective component is of the following form (there are
actually $n$ arrows from $P_i$ to $P_{i+1}$):
 $$ \xymatrix@R=12pt@C=15pt{
    &P_2\ar@{->}[rd]\ar@{.}[rr] &&
    P_4\ar@{->}[rd]\ar@{.}[rr]&& P_6\ar[rd] &\ldots&
    \\
    P_1\ar@{->}[ru]\ar@{.}[rr] &&
    P_3\ar@{->}[ru]\ar@{.}[rr]&&
    P_5\ar[ru]\ar@{.}[rr]&&P_7&\ldots
    \\}
 $$
Starting with $\udim P_1=(0,1)$,
$\udim P_2=(1,n)$  and $\udim I_0=(1,0)$, $\udim I_1=(n,1)$,
one can obtain all the  positive real roots by applying $\Phi^{-i}$ and $\Phi^i$, respectively.

The structure of a regular component gives rise to the following notions. An
indecomposable regular module $X$ is called quasi-simple if
the middle term of the AR sequence ending at $X$ is
indecomposable, i.e. $X$ is on the bottom of a regular component. For
each indecomposable regular module $M$, there is a unique
sequence of irreducible epimorphisms  $M=[r]X\ra [r-1]X\ra\cdots\ra [2]X\ra
[1]X=X$ for some
quasi-simple module $X$ (called quasi-top of $M$) and some $r\geq 1$ (called
quasi-length of $M$). Note that if $M=[r]X$, then the dimension
vector of $M$ is $\sum_{i=0}^{r-1}\udim\tau^i{X}$.  Similarly, there is 
a unique
sequence of irreducible monomorphisms
$Y=Y[1]\ra Y[2]\ra\ldots\ra Y[r-1]\ra Y[r]=M$ with a quasi-simple module $Y$ (called quasi-socle of $M$).

\subsection{Coordinates of positive real roots}

Let $A_0=0$ and $A_i=(\udim P_i)_1$ for $i\geq 1$.
It follows
from the Auslander-Reiten sequences
$0\ra P_i \ra P_{i+1}^n\ra P_{i+2}\ra 0$ that
$A_{i+2}=nA_{i+1}-A_i$ for all $i\geq 1$.
Thus $$A_1=1, \quad A_2=n, \quad A_3=n^2-1,\quad A_4=n^3-2n, \quad A_5=n^4-3n^2+1,\quad\ldots $$
Note that $\udim P_i=(A_{i-1},A_i)$ and each indecomposable preinjective module
has dimension vector $(A_{j+1}, A_j)$ for some $j\geq 0$.

If $n=3$, then the sequence $\{A_i\}_{i\geq 1}$ is
$1,3,8,21,55,\ldots$ and $A_i=F_{2i}$, where $F_j$ is the $j$-th Fibonacci number defined inductively: 
$F_0=0$, $F_1=1$, and $F_{j+2}=F_{j+1}+F_j$.
Some interesting facts of Fibonacci and their representation theory explaination can be found, for example, in 
\cite{C, FR}.
Like Fibonacci numbers, one can find hundreds of 
interesting identities concerning these $A_i$'s. We list several identities those will be used later on.  The proofs are straightforward.

\begin{lemm} \label{equa} Let $\mathcal{K}_n$ be a wild Kronecker quiver with $n\geq 3$ and  $i\geq 1$. Then
   \begin{enumerate}
       \item $A_{i+2}=nA_{i+1}-A_i$
       \item $A_i^2+A_{i+1}^2-nA_iA_{i+1}=1$.
       \item $A_{i+1}^2-A_{i+2}A_i=1$.
       \item $(A_{i+2}+A_{i+1})^2-(A_{i+1}+A_{i})(A_{i+2}+A_{i+3})=n+2$.
       \item $A_iA_{j+k}-A_jA_{i+k}=A_{i-j}A_k$ for all $k\geq 0$ and  $i\geq j$.
  \end{enumerate}
\end{lemm}

\begin{lemm}
Let $n\geq 3$ and  $X$ be an indecomposable regular $\mathcal{K}_n$-module with dimension vector $(a,b)$.
Then $\udim \tau^iX=(A_{2i+1}a-A_{2i}b, A_{2i}a-A_{2i-1}b)$ for each $i\geq 1$.
\end{lemm}

\begin{proof} Recall that $\udim \tau X=(\udim X)\Phi$ where $\Phi$ is the Coxeter matrix. One can finish the proof by induction. 
\end{proof}

\section{quasi-simple modules}

Given a wild quiver and a positive imaginary root $\underline{d}$, are there always quasi-simple modules $M$ with $\udim M=\underline{d}$? This seems to be an unsolved problem, not even for  wild Kronecker quivers.  
In this section we are going to study wild Kronecker quivers, and show the existence of the bricks, i.e.,  modules with trivial endomorphism ring for every imaginary root, which directly implies the existence of quasi-simple modules. Using this we can determine the possible quasi-lengths of the indecomposable modules with a given dimension vector.

\subsection{} Let  $\mathcal{K}_n$ be a wild Kronecker quiver with $n\geq 3$ arrows. One way to obtain a quasi-simple $\mathcal{K}_{n}$-module is to construct bricks. It is straightforward to see that a regular brick is  quasi-simple. 
Another well-known method to obtain quasi-simple modules uses embedding of a subquiver $\mathcal{K}_{n-1}$.  The following proposition can be easily obtained by some elementary calculations using Coxeter matrices.  
We denote by $\Phi_{n-1}$ and $\Phi_{n}$ the Coxeter matrix of $\mathcal{K}_{n-1}$ and $\mathcal{K}_{n}$,
respectively. 

\begin{prop} Fix an embedding of $\mathcal{K}_{n-1}$ in $\mathcal{K}_{n}$.
Let $(c,d)$ be an imaginary root of $\mathcal{K}_{n-1}$ and $M$ an indecomposable
$\mathcal{K}_{n-1}$-module with $\udim M=(c,d)$. Then

   \begin{enumerate}
       \item The vector $(c,d)$ is an imaginary root for $\mathcal{K}_{n}$.
       \item  Let $(a,b)=(c,d)\Phi_{n}$.
              Then $(a,b)$  is no longer a root for
              $\mathcal{K}_{n-1}$.
       \item  The module $M$, as an indecomposable  $\mathcal{K}_{n}$-module,
               is quasi-simple.
       \item  Any non-simple $\mathcal{K}_{n-1}$-module is
               a quasi-simple $\mathcal{K}_{n}$-module.
   \end{enumerate}
\end{prop}

\subsection{Using coordinates of real roots} Using the coordinates of the positive real roots $A_i$ introduced in previous section, one obtains a method to tell for a certain imaginary root  $\underline{d}$ that there are only quasi-simple modules with dimension vector $\underline{d}$.

\begin{lemm} Let $\mathcal{K}_n$ be a wild Kronecker quiver with $n\geq 3$ arrows.
       Let $r\geq 1$ be an odd number and denote by
       $s_r=A_r-A_{r-2}+A_{r-4}+\ldots+(-1)^{\frac{r-1}{2}}A_1$. Then $A_{r+1}=ns_r$.
\end{lemm}

\begin{proof}
We use induction on $r$. This is obvious for $r=1$ since $s_1=A_1=1$ and $A_2=n$.
Now let $r\geq 2$, then
\begin{displaymath}
\begin{array}{rcl}A_{r+3} & = & nA_{r+2}-A_{r+1}\\
              & = & n(A_{r+2}-A_{r}+A_{r-2}+\ldots+(-1)(-1)^{\frac{r-1}{2}}A_1)\\
           & = & ns_{r+2}\\
\end{array}
\end{displaymath}
This finishes the proof.
\end{proof}

Let $X$ be a quasi-simple $\mathcal{K}_n$-module.  We define inductively a sequence
of indecomposable modules as follows:  Let $X_1=X$. If $X_i$ is defined,
then $X_{i+1}$ is defined to be the unique indecomposable module $Y$ such that
there is an irreducible monomorphism $X_i\ra Y$ if $i$ is odd, and
there is an irreducible epimorphism $Y\ra X_i$ if $i$ is even. It is clear that the quasi-length  $\ql(X_i)=i$.

\begin{lemm}\label{dimensionvector} Let $\mathcal{K}_n$ be a wild Kronecker quiver with $n\geq 3$ arrows
 and  $X$ be a quasi-simple module with dimension vector $(a,b)$.
Then
\begin{displaymath}
\udim X_r=\left\{\begin{array}{ll}A_r\udim X_1=A_r(a,b), & $r$\,\ odd;\\
                                  s_{r-1}\udim X_2=A_r(b,nb-a), & $r$\,\ even.\\
          \end{array}\right.
\end{displaymath}
\end{lemm}

\begin{proof}
 This holds obviously for  $r=1$ and $r=2$.
Let $r+1\geq 3$ be odd. Then
\begin{displaymath}
  \begin{array}{rcl}\udim X_{r+1} & = & \udim \tau X_r+\udim X_r-\udim X_{r-1}\\
 & \stackrel{induction}{=}& s_{r-1}(\udim \tau X_2+\udim X_2)-A_{r-1}\udim X_1\\
    & = & s_{r-1}n^2\udim X_1-A_{r-1}\udim X_1 \\
   & = & (nA_r-A_{r-1})\udim X_1\\
  & = & A_{r+1}\udim X_1.
  \end{array}
\end{displaymath}
Let $r+1\geq 4$  be even. Then
\begin{displaymath}
  \begin{array}{rcl}
  \udim X_{r+1} & = & \udim \tau^{-1} X_r+\udim X_r-\udim X_{r-1}\\
 & \stackrel{induction}{=}& A_r(\udim\tau^{-1} X_1+\udim X_1)-s_{r-2}\udim X_2\\
    & = & (A_r-s_{r-2})\udim X_2 \\
   & = & s_r\udim X_2.
  \end{array}
\end{displaymath}
Since $\udim X_2=(nb,n^2b-na)$, we have $s_{r-1}\udim X_2=s_{r-1}n(b,nb-a)=A_r(b,nb-a)$,  if $r$ is even.
\end{proof}

As a direct consequence, the following proposition holds. 

\begin{prop}\label{commondivisor}
Let $\mathcal{K}_n$ be a wild Kronecker quiver with $n\geq 3$ arrows and $(a,b)$ be an imaginary root.   
 If $A_t$ is not a common divisor of  $a$ and $b$ for any $t\geq 2$,
           then every indecomposable $\mathcal{K}_n$-module $M$ with $\udim M=(a,b)$ is quasi-simple.
\end{prop}

\begin{example} Let $n=3$. Then 
indecomposable $K_3$-modules with dimension vectors $(100,100)$, $(100, 200)$ and $(100, 96)$ are all quasi-simple, because $3,8,21,55$ are not divisor of $100$.   
\end{example}

\begin{ques}Let $(a,b)$ be an imaginary root for $\mathcal{K}_n$. If $M$ is an indecomposable module with dimension vector $\udim M=(a,b)$ and quasi-length $r$, then  $A_r$ is a common divisor of $a$ and $b$ by Lemma \ref{dimensionvector}. Conversely, if $(a,b)=A_r(a',b')$,  is there  an indecomposable $\mathcal{K}_n$-module $M$ with dimension vector $\udim M=(a,b)$ and quasi-length $\ql(M)=r$?  Obviously, if there is a quasi-simple module $N$ with dimension vector $\udim N=(a', b')$, then we may take  $M=N_r$. So the question is:  Given an imaginary root, is there always a quasi-simple module?   
\end{ques}

Now we start to show the existence the bricks, which in turn gives an affirmative answer to the previous question.  

\begin{theo}\label{bricks} Let $\mathcal{K}_n$ be a wild Kronecker quiver with $n\geq 3$. Given an imaginary root $\underline{d}$,  there is 
always a brick $M$, (which is thus  quasi-simple),  with dimension vector $\udim M=\underline{d}$.   
\end{theo}

\begin{proof} Let $\underline{d}=(rb+s,b)$ with $0\leq r\leq  n-1$, $0\leq s<b$.  We construct the non-zero (block) matrices $\alpha_{i, 1\leq i\leq n}$ for the $n$ arrows, such that a representation of $\mathcal{K}_n$ with dimension vector $\underline{d}$ and these matrices is a brick.  We denote by $I$ the $b\times b$ identity matrix  and by $0_m$ the $m\times b$ matrix with zero entries for $m>0$.   
For a matrix $M$ we denote by $F(M)$ the matrix obtained from $M$ by deleting the last row and inserting
$0_1$ as the new first row.

\begin{enumerate}
\item \underline{Case $r=1, s=0$}.   Let $\alpha_1=J$, $\alpha_2=I$ and $\alpha_3=J'$, where $J$ is the $b\times b$ Jordan matrix  with eigenvalue $0$ and $J'$ is the transpose of $J$.

\item \underline{Case $r=1, 0<s<b$}.  Let $\alpha_i$ be the following $(b+s)\times b$ matrices: 
\[\alpha_1=\left(\begin{array}{c} I\\ 0_s\end{array}\right),\,\ \alpha_2= \left(\begin{array}{c} 0_s\\ I\end{array}\right),\,\ \alpha_3= F(\alpha_1)=\left(\begin{array}{c} 0_{1}\\I\\ 0_{s-1}\end{array}\right).\]  In case $s=1$, we have $\alpha_2=\alpha_3$ and thus two non-zero matrices are sufficient. 

\item \underline{Case $2\leq r\leq n-1 , s=0$} . Let $\alpha_i$ be the following $r\times 1$ block matrices: 
\[
\alpha_1=\left(\begin{array}{c} I\\ 0_b\\ \vdots \\0_b \end{array}\right),\,\
\alpha_{i,(2\leq i\leq r)}=
\begin{blockarray}{cccl}
\begin{block}{(ccc)l} 
& \vdots &  &  \\  & 0_b &  & \\ & I &  &\matindex{i}\\ & 0_b &  & \\ & \vdots &  &  \\ 
\end{block}  
\end{blockarray}, \,\
\alpha_{r+1}=F(\alpha_1).
\]

\item \underline{Case  $2\leq r\leq n-2 (n\geq 4), 0<s<b$}. 
Let $\alpha_i$ be the following  $(r+1)\times 1$ block matrices:
\[
\alpha_1=\left(\begin{array}{c} I\\ 0_b\\  \vdots \\ 0_b \\ 0_s\end{array}\right),\,\
\alpha_{i,(2\leq i\leq r)}=
\begin{blockarray}{cccl}
\begin{block}{(ccc)l} 
& \vdots &  &  \\ & 0_b &  & \\ & I &  &\matindex{i}\\ & 0_b &  & \\ & \vdots &  &  \\ & 0_b &  & 
\\& 0_s & &\\
\end{block}  
\end{blockarray}, \,\
\alpha_{r+1}=\left(\begin{array}{c} 0_{s}\\ 0_b\\ \vdots \\ 0_b  \\ I \end{array}\right),\,\
\alpha_{r+2}=F(\alpha_1).
\]

\item \underline{Case $r=0$, i.e., $\underline{d}=(s,b)=(s,ts+c)$,  $1\leq t\leq n-2$,  $c\leq s$}.  This is the dual case of (1)--(4).  

\item \underline{Case $r=n-1, 0<s<b$, i.e.,  $\underline{d}=((n-1)b+s,b)$}. 
Consider the imaginary roots  $(a_i,b_i)=\underline{d_i}=\underline{d}\,\Phi_n^{-i}$ and choose the minimal $j$ such that $a_{j+1}\leq b_{j+1}$. If there is some $i<j+1$ such that 
$\underline{d_i}=(a_i,b_i)=(r_ib_i+s_i, b_i)$ with $1\leq r_i\leq n-2$ and $0\leq s_i\leq b_i$, then following (1)--(4) let $N$ be a brick with dimension vector $\underline{d_i}$ and $M=\tau^iN$.  Otherwise, $\underline{d_j}=(a_j,b_j)=((n-1)b_j+s_j, b_j)$ with $0< s_j<b_j$  and $\underline{d_{j+1}}=(a_{j+1}, b_{j+1})=\underline{d_j}\,\Phi_n^{-1}=(b_j-s_j, (n-1)(b_j-s_j)-s_j)$. Since 
$a_{j+1}\leq b_{j+1}$, $\underline{d_{j+1}}=(b_j-s_j, r(b_j-s_j)+c)$ for some $1\leq r\leq n-2$ and $0\leq c<b_j-s_j$. Thus by (5) there is a brick $N$ with $\udim N=\underline{d_{j+1}}$. Take $M=\tau^{j+1}N$.  
In both of above situations, $M$ is a brick with $\udim M=\underline{d}$.

\item \underline{Case $r=0$, $\underline{d}=(s,b)=(s,(n-1)s+c)$, $c<s$}. Duality of (6). 
\end{enumerate}
The proof is completed. 
 \end{proof}

The following theorem is a  direct consequence of previous discussion.  
\begin{theo}\label{ql}
Let $\mathcal{K}_n$ be a wild Kronecker quiver with $n\geq 3$ arrows and  $\underline{d}=(a,b)$ be an imaginary root.  Then $a$ and $b$ have a common divisor $A_r$, i.e., $(a,b)=A_r(a',b')$ if and only if there exists an indecomposable module $M$ with $\udim M=\underline{d}$ and quasi-length $r$. 
\end{theo}

\section{Indecomposable modules with the same length in a regular component}

In a regular component of the AR quiver of a wild quiver, the indecomposable modules 
are uniquely determined by their dimension vectors \cite{Z}.  It follows that
given a natural number $d$, there are, in a regular component, only finitely many indecomposable modules with length $d$.  Let $\mathcal{K}_n$ be a wild Kronecker quiver with $n\geq 3$ arrows.
 Let $\beta_{\mathcal{K}_n,\mathcal{C}}(d)$ denote the number of the indecomposable modules in a regular component $\mathcal{C}$ of the AR-quiver with length $d$ and $\beta_{\mathcal{K}_n}=\sup_{\mathcal{C},d}\,\beta_{\mathcal{K}_n, \mathcal{C}}(d)$.  We will show in this section $\beta_{\mathcal{K}_n}\leq 2$ and study the dimension vectors when the equality holds.

\subsection{}Let $\mathcal{K}_n$ be a wild Kronecker quiver with $n\geq 3$.
For each $i\geq 0$, let $B_{2i}=A_i$ and $B_{2i+1}=B_{2i+2}-B_{2i}$.  Thus
$$B_0=0,\quad B_1=1, \quad B_2=1,\quad B_3=n-1,\quad B_4=n,\quad B_5=n^2-n-1,\quad B_6=n^2-1,\ldots$$
It is easily seen  that $B_i<B_{i+1}$ and
$B_{2i-1}+(n-2)B_{2i}=B_{2i+1}$.
If $n=3$, then $B_{2i+1}$ is exactly the $2i+1$-th Fibonacci number $F_{2i+1}$.
For each $i\geq 0$, $(B_{2i}, B_{2i+2})=(A_i,A_{i+1})$ is a  positive real roots.
By applying the Coxeter transformation $\Phi$ and Lemma \ref{dimensionvector}, we can show
that $(B_{2i-1},B_{2i+1})$ is an imaginary root for every $i\geq 1$. The dimension vectors of 
the indecomposable modules in a regular components containing this kind of imaginary roots can be easily understood. 

\begin{prop}
Let $M$ be an indecomposable module in a regular component.
Suppose that $\ql(M)=r+1\geq 2$ and the quasi-top of $M$ is $X$, i.e. $M=[r+1]X$.
\begin{enumerate}
   \item $\udim M=(m,m)$ if and only if $\udim X=b(B_{2r-1},B_{2r+1})$ and $m=bA_{r+1}$.
   \item  $\udim M=(m(n-1),m)$ if and only if  $\udim X=b(B_{2r-3}, B_{2r-1})$ ($(b,b)$
              if $r=1$) and $m=bA_{r+1}$.
\end{enumerate}
\end{prop}
The proof is somehow straightforward.   We also refer to \cite{C} for the $n=3$ case.

\subsection{ Some  (in)equalities} To study the indecomposable modules with the same length in a regular component we need the following (in)equalities.

The following well-known formula of Fibonacci numbers holds
for arbitrary $r$ and $s$:
$F_{r+s}=F_rF_{s+1}+F_{r-1}F_s$.
Similarly, we have 
\begin{lemm}
For each $s, t \geq 1$, $B_{2s+2t}=B_{2s}B_{2t+1}+B_{2s-1}B_{2t}$.
\end{lemm}

\begin{remark}Note that  if $n\geq 4$,  the formula in this lemma
holds only for $B_{r+s}$, where $r+s$ is an even number. For example, let $r=3,s=4$,
then $B_7=A_4-A_3=n^3-n^2-2n+1$ and $B_3B_5+B_2B_4=(n-1)(n^2-n-1)+n=n^3-2n^2+n+1$.  They are
not equivalent since $n\geq 4$.
\end{remark}

\begin{lemm}\label{inequal0}
  For every $2\leq t<r$, $n=\frac{A_2}{A_1}> \frac{A_{t+1}}{A_t}>\frac{A_{r+1}}{A_r}>n-1$.
\end{lemm}

\begin{proof}
  It is sufficient to show that $\frac{A_{r+1}}{A_r}>\frac{A_{r+2}}{A_{r+1}}$
  for every $r\geq 1$.  Since
  $A_{r+1}^2-nA_{r+1}A_{r+2}+A_{r+2}^2=1$ and $A_{r+2}=nA_{r+1}-A_r$,
  we have $A_{r+  1}^2-nA_{r+1}A_{r+2}+(nA_{r+1}-A_r)A_{r+2}>0$.
  It follows that $A_{r+1}^2>A_{r}A_{r+2}$.
\end{proof}

Using the above two lemmas, we can show the following inequality.
\begin{lemm} \label{inequa1}

Let $j,k,s,t>0$. Then $\frac{A_{j+k}}{A_j}>\frac{A_{s+t}}{A_s}$ if $k>t$, or $k=t$ and $j<s$. 
In particular,  $\frac{A_i}{A_j}=\frac{A_r}{A_s}$ if and only if $i=r$ and $j=s$.
\end{lemm}

\begin{proof}
Let $i=j+k$ and $r=s+t$.
Then $$A_{j+k}=B_{2j+2k}=B_{2j}B_{2k+1}+B_{2j-1}B_{2k},$$
$$A_{s+t}=B_{2s+2t}=B_{2s}B_{2t+1}+B_{2s-1}B_{2t}.$$
Thus
$$\frac{A_{j+k}}{A_j}=\frac{B_{2j+2k}}{B_{2j}}=B_{2k+1}+\frac{B_{2j-1}}{B_{2j}}B_{2k},$$
$$\frac{A_{s+t}}{A_s}=\frac{B_{2s+2t}}{B_{2s}}=B_{2t+1}+\frac{B_{2s-1}}{B_{2s}}B_{2t}.$$
If  $k>t$, then $$\frac{A_{s+t}}{A_s}<B_{2t+1}+B_{2t}=B_{2t+2}<B_{2k+1}<\frac{A_{j+k}}{A_j}.$$
Suppose that $k=t$ and $j<s$. Then $\frac{A_{j+k}}{A_j}$ is the product
$\frac{A_{j+k}}{A_{j+k-1}}\frac{A_{j+k-1}}{A_{j+k-2}}\cdots\frac{A_{j+1}}{A_j}$ with $k=t$ factors
and so is $\frac{A_{s+t}}{A_s}$. Thus $\frac{A_{j+k}}{A_j}>\frac{A_{s+t}}{s}$ by previous lemma. 
\end{proof}

\begin{lemm}\label{inequa2}
For any $i,j,r,s,t\geq 1$,  $\frac{A_{i+j}}{A_s}\neq\frac{A_i}{A_t}+\frac{A_j}{A_r}$.
\end{lemm}

\begin{proof} Suppose for a contradiction that the equality holds. 

\underline{ Case   $i+j=s$.} Since  $\frac{A_{i+j}}{A_s}=1$, both $\frac{A_i}{A_t}$ and $\frac{A_j}{A_r}$ are smaller than $1$. 
In particular $i< t$, $j< r$ and $\frac{A_i}{A_t}\leq \frac{A_{t-1}}{A_t}<\frac{1}{2}$,  $\frac{A_j}{A_r}\leq\frac{A_{r-1}}{A_r}<\frac{1}{2}$.
This is a contradiction. 

\underline {Case  $i+j>s$.}  Let $i+j=s+a$ for some $a\geq 1$. Because $\frac{A_{i+j}}{A_s}>2$ we may assume without loss of generality that $\frac{A_i}{A_t}>1$ and $i=t+b$ for some 
$b\geq 1$.  If $a\leq b$, then $\frac{A_{i+j}}{A_s}=\frac{A_{i+j}}{A_{i+j-a}}<\frac{A_i}{A_{i-b}}=\frac{A_i}{A_t}$, which is impossible. 
Therefore, $a>b$. It follows that 
\begin{displaymath}
\begin{array}{rcl}
& & \frac{A_{i+j}}{A_s}-\frac{A_i}{A_t}=\frac{A_{s+a}}{A_s}-\frac{A_{t+b}}{A_t}\\
&=& B_{2a+1}+\frac{B_{2s-1}}{B_{2s}}B_{2a}-(B_{2b+1}+\frac{B_{2t-1}}{B_{2t}}B_{2b})\\
&>& B_{2a+1}-B_{2b+2}\left\{ \begin{array}{ll} =B_{2a+1}-B_{2a}=(n-3)B_{2a}+B_{2a-1}, & a=b+1;\\
                                               \geq B_{2a+1}-B_{2a-2}=(n-3)B_{2a}+2B_{2a-1} & a\geq b+2.\\\end{array}\right.
\end{array}
\end{displaymath}

Since $B_{2a-1}\geq 2$, it follows that  $\frac{A_{i+j}}{A_s}-\frac{A_i}{A_t}>1$. Then $j=r+c$ for some $c\geq 1$ and $\frac{A_{j}}{A_r}=B_{2c+1}+\frac{B_{2r-1}}{B_{2r}}B_{2c}\leq B_{2c+2}$.
Similarly,  $a>c$. 
On the other hand, because $B_{2c+2}\leq B_{2a}<2B_{2a-1}$ and $B_{2a-2}<B_{2a-1}$, $\frac{A_{i+j}}{A_s}-\frac{A_i}{A_t}=\frac{A_j}{A_r}$ implies
$n=3$ and  $a=b+1=c+1$, that is 
\[B_{2a+1}+\frac{B_{2s-1}}{B_{2s}}B_{2a}=B_{2a-1}+\frac{B_{2t-1}}{B_{2t}}B_{2a-2} + B_{2a-1}+\frac{B_{2r-1}}{B_{2r}}B_{2a-2}.\] 
It follows that 
\[B_{2a-2}+\frac{B_{2s-1}}{B_{2s}}B_{2a}=(\frac{B_{2t-1}}{B_{2t}}+\frac{B_{2s-1}}{B_{2s}})B_{2a-2}.\]
The left hand side is greater than $(2\frac{B_{2s-1}}{B_{2s}}+1)B_{2a-2}>2B_{2a-2}$. The right hand side is, however, smaller than $2B_{2a-2}$.  This is a contradiction. 

\underline{Case  $i+j<s$. }  Let $i+j+a=s$, $i+b=t$ and $j+c=r$ for some $a, b, c\geq 1$. It follows that $a\leq b$ and $a\leq c$. Otherwise, we would have (Lemma \ref{inequa1})
$\frac{A_{i+j}}{A_s}=\frac{A_{i+j}}{A_{i+j+a}}<\frac{A_i}{A_{i+b}}=\frac{A_i}{A_t}$, which is impossible.

If $a=b$, then \[\frac{A_j}{A_r}=\frac{A_{i+j}}{A_s}-\frac{A_i}{A_t}=\frac{A_{i+j}}{A_{i+j+a}}-\frac{A_i}{A_{i+b}}=
\frac{A_{i+j}A_{i+a}-A_{i+j+a}A_i}{A_{i+j+a}A_{i+a}}\stackrel{\textrm{Lemma \ref{equa}(5)}}{=}\frac{A_{j}A_{a}}{A_{i+j+a}A_{i+a}}.\]
This is impossible by Lemma \ref{inequal0}. Thus we have $a< b$ and similarly,  $a< c$.

If $a\leq b-1$, $a\leq c-1$, then \[\frac{A_{i+j}}{A_{i+j+a}}=\frac{A_i}{A_{i+b}}+\frac{A_j}{A_{j+c}}
                                    =\frac{A_i}{A_{i+a}}\frac{A_{i+a}}{A_{i+b}}+\frac{A_j}{A_{j+a}}\frac{A_{j+a}}{A_{j+c}}<\frac{1}{2}(\frac{A_i}{A_{i+a}}+\frac{A_j}{A_{j+a}})\]
Therefore, \[\frac{A_{i+j}}{A_{i+j+a}}-\frac{A_i}{A_{i+a}}< \frac{A_j}{A_{j+a}}-\frac{A_{i+j}}{A_{i+j+a}} < 0\]
This is again a contradiction.  

Thus   $\frac{A_{i+j}}{A_s}\neq\frac{A_i}{A_t}+\frac{A_j}{A_r}$ for any $i,j,r,s,t\geq 1$.
\end{proof}

\subsection{Upper bound} Let $\mathcal{K}_n$ a wild Kronecker quiver with $n\geq 3$ arrows. We can give an upper bound of the number of indecomposable modules with a given length in an arbitrary regular component. 

\begin{theo}\label{bound} 
Let  $\mathcal{K}_n$ be a wild Kronecker quiver with $n\geq 3$. Then  $\beta_{\mathcal{K}_n}\leq 2$,  that is, for any natural number $d$, there are, in any regular component $\mathcal{C}$ and,  at most two (up to isomorphism) indecomposable modules  with length $d$.
\end{theo}
\begin{proof} Let $\mathcal{C}$ be a regular component and $M$ an indecomposable module with $M=X_t$, where $X$ is a quasi simple module  with dimensional vector $(a,b)$.
Let $N=(\tau^j X)_s$ and $L=(\tau^i X)_r$. Assume for a contradiction that $|M|=|N|=|L|$. \\
\underline{Case $r$ even and $s$, $t$  odd}. The dimensional vectors are 
\[\udim M = A_t(a,b), \udim N=A_s(A_{2j+1}a-A_{2j}b, A_{2j}a-A_{2j-1}b),\] \[\udim L= A_r(A_{2j}a-A_{2j-1}b), n(A_{2j}a-A_{2j-1}b)-(A_{2j+1}a-A_{2j}b).\]
The modules are of the same length implies that   
\begin{displaymath}
\begin{array}{rcl}
 A_t(a+b) & = & A_s((A_{2j+1}a-A_{2j}b)+(A_{2j}a-A_{2j-1}b))\\
          & = & A_r((A_{2i}a-A_{2i-1}b+n(A_{2i}a-A_{2i-1}b)-(A_{2i+1}a-A_{2i}b)).
\end{array} 
\end{displaymath}
It follows that
\[\frac{A_s(A_{2j+1}+A_{2j})-A_t}{A_s(A_{2j}+A_{2j-1})+A_t}=\frac{b}{a}
 =\frac{A_r(A_{21}+A_{2i-1})-A_t}{A_r(A_{2i-1}+A_{2i-2})+A_t}.\]
Then using the formula Lemma \ref{equa} (5) (for $k=1,2$) we obtain \[A_rA_sA_{2i-2j-1}+A_sA_tA_{2j}=A_rA_tA_{2i-1}\]
and thus \[\frac{A_{2i-2j-1}}{A_t}+\frac{A_{2j}}{A_r}=\frac{A_{2i-1}}{A_s}.\]

The equalities obtained from all possibilities ($r,s,t$ being even or odd numbers) are listed in the follow tabular.  

\begin{center}
\begin{tabular}{|c|c|c|}
\hline 
 Even & Odd & Equality $\frac{A_{2i-2j+x}}{A_t}+\frac{A_{2j+y}}{A_r}=\frac{A_{2i+z}}{A_s}$ with \\
 
 & & $x, y, z=$\\
\hline
 
  $r,s,t$ & &  $ \begin{array}{lll}0, & 0, & 0\end{array} $\\

& $r,s,t$  &  $ \begin{array}{lll}0, & 0, & 0\end{array} $\\

 $r$  & $s,t$ & $ \begin{array}{lll}-1, & 0, & -1\end{array}$\\

$s$ & $r,t$ & $\begin{array}{lll}1, & -1, & 0\end{array}$\\

$t$ & $r,s$ & $ \begin{array}{lll}0, & 1, & 1\end{array}$\\

$s,t$, & $r$ & $ \begin{array}{lll}1, & 0, & 1\end{array}$\\

$r, t$ & $s$ & $ \begin{array}{lll}-1, & 1, & 0\end{array}$\\

$r,s$ & $t$ & $ \begin{array}{lll}0, & -1, & -1\end{array}$ \\

\hline
\end{tabular}
\end{center}

However, the equalities could not hold by Lemma \ref{inequa2}.
The proof is completed. 
\end{proof}

\subsection{} Indecomposable modules in a regular component with the same length occur  everywhere in the following sense: 

\begin{theo}\label{samelength}Let  $\mathcal{K}_{n}$ be a wild Kronecker quiver with $n\geq 3$ arrows.
Then for each pair of natural numbers $r,s$, there are always
regular components $\mathcal{C}$ containing indecomposable
modules $M$ and $N$ with quasi-lengths $r,s$, respectively
such that $|M|=|N|$.
\end{theo}

\begin{proof}
Without loss of generality, we may assume $s\geq r$. We need to find
some quasi-simple $\mathcal{K}_{n}$-module $X$ such that $|X_s|=|(\tau^iX)_r|$
for some $i\geq 1$.
Suppose that it is the case, and $\udim X=(c,d)$.
Then $\udim \tau^iX=(a,b)=(c,d)\Phi^i=(A_{2i+1}c-A_{2i}d,A_{2i}c-A_{2i-1}d)$.

We first assume $r$ and $s$ are odd.
Then $\udim X_s=A_s(c,d)$ and $\udim (\tau^iX)_r=A_r(a,b)$
and  $|X_s|=|(\tau^iX)_r|$ implies that
  $$A_r(A_{2i+1}+A_{2i})c-A_r(A_{2i}+A_{2i-1})d=A_s(c+d).$$
It follows that
  $$\frac{A_r(A_{2i}+A_{2i-1})+A_s}{A_r(A_{2i+1}+A_{2i})-A_s}=\frac{c}{d}.$$
We look for some $i$ satisfying the  following inequalities
  $$\frac{1}{n-1}\leq
  \frac{A_r(A_{2i}+A_{2i-1})+A_s}{A_r(A_{2i+1}+A_{2i})-A_s}=\frac{c}{d}\leq n-1.$$
Note that
\begin{displaymath}
  \begin{array}{cl}
      & \frac{A_r(A_{2i}+A_{2i-1})+A_s}{A_r(A_{2i+1}+A_{2i})-A_s}\leq n-1\\
     \Longleftrightarrow & \frac{n A_s}{A_r}\leq (n-1)A_{2i+1}-A_{2i}+(n-1)A_{2i}-A_{2i-1}\\
     \Longleftrightarrow & \frac{n A_s}{A_r}\leq (n A_{2i+1}-A_{2i})+(n+A_{2i}-A_{2i-1})-A_{2i+1}-A_{2i}\\
     \Longleftrightarrow & \frac{n A_s}{A_r}\leq A_{2i+2}+A_{2i+1}-A_{2i+1}-A_{2i}\\
     \Longleftrightarrow & \frac{n A_s}{A_r}\leq A_{2i+2}-A_{2i}.\\
   \end{array}
\end{displaymath}
Similarly, we have $\frac{n A_s}{A_r}>A_{2i}-A_{2i-2}$.
Since $s\geq r$, there is some $i$ such that the following inequalities hold
   \begin{equation}
     A_{2i}-A_{2i-2}\leq \frac{n A_s}{A_r}\leq A_{2i+2}-2A_{2i}.
   \end{equation}
Since the sequence $\{A_{i+1}-A_{i-1}\}_{i\geq 1}$ is increasing,
there is some $i$ such that $\frac{1}{n-1}\leq \frac{c}{d}\leq n-1$.
We fix such an index $i$ and let $c,d$ process no common divisors $A_t$ for any $t\geq 2$.
Let $X$ be an indecomposable $\mathcal{K}_n$-module with dimension
vector $\udim X=(c,d)$. Thus $X$ is  quasi-simple by Corollary \ref{commondivisor}.
Let $N=X_s$ and $M=(\tau^iX)_r$. Then $|M|=|N|$.

Now we assume $r$ is odd, whereas $s$ is even.
Then $\udim X_s=A_s(d,nd-c)$ and $\udim (\tau^iX)_r=A_r(a,b)$.
Thus $|X_s|=|(\tau^iX)_r|$ implies
$$A_r(A_{2i+1}+A_{2i})c-A_r(A_{2i}+A_{2i-1})d=A_s(d+nd-c).$$
and
$$\frac{A_r(A_{2i}+A_{2i-1})+(n+1)A_s}{A_r(A_{2i+1}+A_{2i})+A_s}=\frac{c}{d}.$$
Again we check the possibility that
$$\frac{1}{n-1}\leq \frac{A_r(A_{2i}+A_{2i-1})+(n+1)A_s}{A_r(A_{2i+1}+A_{2i})+A_s}=\frac{c}{d}\leq n-1.$$
  \begin{equation}
     \left\{\begin{array}{l}
          \frac{(n^2-2) A_s}{A_r}\geq A_{2i}-A_{2i-2}\\
          \frac{2A_s}{A_r}\leq A_{2i+2}-A_{2i}
            \end{array}
     \right.
  \end{equation}
Since $n\geq 3$, there exists some $i$ such that
the inequalities in (2) hold.
Thus $(c,d)$ is an imaginary root. We may assume that $c,d$ have no common divisor $A_t$ for any $t\geq 2$.
Let $X$ be an indecomposable module with dimension vector $(c,d)$. Then $X$
is quasi-simple by Corollary \ref{commondivisor}. Let $N=X_s$ and
$M=(\tau^i X)_r$. Then $|M|=|N|$.

The other cases follow similarly.
\end {proof}

\begin{example} 
Let $n=3$, $r=1$ and $s=2$.  We need to find some $i$ such that
$21=(n^2-2)A_2\geq A_{2i}-A_{2i-2}$. The only possibilities are $i=1$ and
$i=2$.  If  $i=1$, then
$\frac{A_1(A_{2i}+A_{2i-1})+(n+1)A_s}{A_1(A_{2i}+A_{2i+1})+A_s}
=\frac{8}{7}$.  Let $X$ be a quasi-simple module with dimension vector
$\udim X=(c,d)=(8,7)$.  Then $\udim X_2=(21,39)$, $\udim \tau X=(43,17)$
and thus $|X_2|=|\tau X|$.
In the case of  $i=2$, we may choose $(c,d)=(41,79)$
and let $X$ be a quasi-simple
module with dimension $\udim X=(c,d)$. Then $\udim X_2=(237,588)$ and
$\udim \tau^2 X=(596,229)$ and thus $|X_2|=|\tau^2X|$.
\end{example}

\subsection{Indecomposable modules in the same Coxeter orbit with the same length}

Let $\mathcal{K}_n$ be a wild Kronecker quiver with $n\geq 3$.
We have seen that given any pair of natural numbers $r,s$, there are always  regular
components containing two non-isomorphic
indecomposable modules with quasi-length $r,s$, respectively, and with the same length.  In this section,
the special case that $r=s$ will be studied in detail. The case $n=3$ was studied in \cite{C} with
the help of Fibonacci numbers.
Let $\Phi$ be the Coxeter matrix of  $\mathcal{K}_n$.

\begin{lemm}Let $(a,b)$ be a vector and $(c,d)=(a,b)\Phi^i$ for some $i\geq 0$.
Suppose that $c+d=a+b$. Let $(a',b')=(a,b)\Phi$ and $(c',d')=(c,d)\Phi^{-1}$.
Then $a'+b'=c'+d'$.
\end{lemm}

\begin{proof}
This follows by direct calculation using $A^2_i-A_{i-1}A_{i+1}=1$.
\end{proof}

If $\mathcal{C}$ is a regular component and $M=X[r]$ is an indecomposable
module of quasi-length $r$ with quasi-socle $X\in\mathcal{C}$,
the mesh-complete
full subquiver of $\mathcal{C}$, defined by the vertices $\tau^{-i}X[l]$ with
$i\geq 0$ and $i+l\leq m$ is called a wing, and denoted by $\mathcal{W}(M)$. For an
indecomposable regular module $M$, we denote by $\qs(M)$ and $\qt(M)$ the
quasi-socle and quasi-top of $M$, respectively.

\begin{theo} \label{sameorbit} Let $\mathcal{K}_n$ be a wild Kronecker quiver with $n\geq 3$. Let $\mathcal{C}$ be a regular component and $M$ and $N$ be indecomposable modules in $\mathcal{C}$ such that $|M|=|N|$.
Then the following are equivalent:
\begin{enumerate}
 \item $M$ and $N$ are in the same $\tau$-orbit, i.e. $\ql(M)=\ql(N)$.
  \item  $\mathcal{C}$  contain a quasi-simple module with dimension vector $(m,m)$ or $(m,(n-1)m)$. 
 \end{enumerate}
\end{theo}

\begin{proof}
(1) $\Rightarrow$ (2). If  $|M|=|\tau^iM|$, then by above lemma, we have $|\tau M|=|\tau^{i-1}M|$. Thus we may
assume, without loss of generality, that
$|\tau M|=|M|$ if $i$ is odd, or $|\tau^2 M|=|M|$ if $i$ is even.
Note that $\udim \tau^iM=(A_{2i+1}x-A_{2i}y, A_{2i}x-A_{2i-1}y)$. It follows that
$$\frac{A_{2i+1}+A_{2i}-1}{A_{2i}+A_{2i-1}+1}=\frac{y}{x}.$$
We first assume that $|\tau M|=|M|$.  Thus $$\frac{y}{x}=\frac{A_3+A_2-1}{A_2+A_1+1}=\frac{n^2-1+n-1}{n+1+1}=n-1.$$
Thus $(x,y)=(m,(n-1)m)$ for some $m\geq 1$.
If $|\tau^2 M|=|M|$, then $$\frac{y}{x}=\frac{A_5+A_4-1}{A_4+A_3+1}=\frac{n^2-n-1}{n-1}.$$ Since $(n-1,n^2-n-1)\Phi=(1,1)$, we have $\udim\tau M=(m,m)$ for some $m\geq 1$.

Let $(m,m)$ be a dimension vector of an indecomposable regular module $M$ with quasi-length $\ql(M)=r\geq 1$. Assume that $M=X_r$ for some quasi-simple module $X$ with dimension vector $(a,b)$.
If $r$ is odd, then  $(m,m)=A_r(a,b)$ and $a=b$.
If $r$ is even, then $(m,m)=A_r(b,nb-a)$. Thus $A_rb=m$ and $nm-A_ra=m$. Therefore,
$(m,m)=A_r((n-1)b,b)$ and $a=(n-1)b$.
Similarly, if $(m,(n-1)m)$ is a dimension vector of an indecomposable module $M$ with quasi-length $\ql(M)=r\geq 1$. Let $M=X_r$ for some quasi-simple module $X$ with dimension vector $(a,b)$. Then
$(m,(n-1)m)=A_r(a,(n-1)a)$ and $b=(n-1)a$
if $r$ is odd, and $(m,(n-1)m)=A_r(a,a)$ and $a=b$ if $r$ is even.

(2) $\Rightarrow$ (1).
Note that $\mathcal{C}$ contains a quasi-simple module $Z$ with dimension
$(x,x)$ or $(x,(n-1)x)$ for some $a\geq 1$.
We consider the case $\udim Z=(x,x)$, and without loss of generality, we may assume $x=1$.
The other case follows similarly.

Let $\ql(M)=r>t=\ql(N)$. Thus $M=X_r$ and $N=Y_t$ for some quasi-simple modules
$X$ and $Y$ with dimensional vectors $(a,b)$ and $(c,d)$, respectively.
Since $\mathcal{C}$ is symmetric, we may assume $Z=\tau^i X=\tau^jY$ with $i,j>0$.
We may also assume that the wings $\mathcal{W}(M)$ and  $\mathcal{W}(N)$ intersect empty.

Assume $r>t$ are both odd numbers.
Then $\udim M=A_r(a,b)$ and $\udim N=(A_t(c,d)$.  Then $A_r(a+b)=A_t(c+d)$ by assumption and thus $a+b<c+d$ since $r>t$. 
Thus $0\leq i<j$.  Let $M'=\qt(M)$ and $N'=\qs(N)$. then $M'=\tau^lN'$ for some $l>0$.
Note that for each $s\geq 1$, $2\sum_{l=0}^s\udim \tau^{-l}Z<\udim \tau^{-(l+1)}Z$. 
This follows by induction and the fact $\udim\tau^{-1}Z=(n-1,n^2-n-1)>(n-1,n-1)$. 
Therefore, $\udim M<2\sum_{l=0}^k\tau^{-i}Z<\udim N'<\udim N$ where $k>0$ is the natural number with $\tau^k M'=Z$. This is a contradiction.

Therefore,  we have $\ql(M)=\ql(N)$ and thus $\udim M=(e,f)$ if and only if $\udim (N)=(f,e)$.
\end{proof}

\section{Regular components with common sets of dimension vectors}

Let $\mathcal{K}_n$ be a wild Kronecker quiver with $n\geq 3$. For a regular component $\mathcal{C}$ of the AR-quiver, let $\udim \mathcal{C}$ denote the set of the dimension vectors of the indecomposable modules in $\mathcal{C}$.  In this section, we will discuss when $\udim\mathcal{C}=\udim \mathcal{D}$ for different regular component $\mathcal{C}$ and $\mathcal{D}$.
The following lemma is straightforward.
\begin{lemm}\label{compare} Let $\mathcal{K}_n$ be a wild Kronecker quiver with $n\geq 3$ and $(a,b)$ an imaginary root. 
\begin{enumerate}

\item If $(c,d)=(a,b)\Phi^{-1}=(nb-a,(n^2-1)b-na)$,
then $\frac{d}{c}>\frac{nb-a}{b}>\frac{b}{a}$.
\item If  $X$ is a quasi-simple module with dimension vector $(a,b)$, then
$\frac{b}{a}=\frac{(\udim X)_2}{(\udim X)_1}<\frac{(\udim X_2)_2}{(\udim X_2)_1}<\frac{(\udim \tau^{-1}X)_2}{(\udim \tau^{-1}X)_1}$.
\end{enumerate}
\end{lemm}

The indecomposable modules in a regular component are uniquely determined by their dimension vectors.   
Next it will be shown that the set of the dimension vectors of the indecomposable modules in a regular component is uniquely determined by a Coxeter orbit and quasi-length, and also uniquely determined by two Coxeter orbits. . 

\begin{theo}\label{twoorbits}
Let $n\geq 3$ and $\mathcal{C}$ and $\mathcal{D}$ be two different regular components of a wild Kronecker quiver $\mathcal{K}$.
Then the following are equivalent:
\begin{enumerate}

\item  There are quasi-simple
modules $X\in\mathcal{C}$ and $Y\in\mathcal{D}$ such that $\udim X=\udim Y$.
\item There are indecomposable modules  $X\in\mathcal{C}$ and $Y\in\mathcal{D}$ with $\udim X=\udim Y$ and $\ql(X)=\ql(Y)$.
\item  $\udim\mathcal{C}=\udim \mathcal{D}$.
\item  There are indecomposable modules $M,M'\in\mathcal{C}$ and $N,N'\in\mathcal{D}$
with $\ql(M)\neq \ql(M')$ and $\ql(N)\neq \ql(N')$ such that $\udim M=\udim N$,
$\udim M'=\udim N'$.
\end{enumerate}
\end{theo}

\begin{proof}(2)$\Longrightarrow$ (1)\quad  Let, for example,
$r=\ql(X)=\ql(Y)>1$ be even.  Then $\udim X=A_r(b,nb-a)$ and $\udim Y=A_r(d,nd-c)$
where $X',Y'$ are quasi-simple modules with $\udim X'=(a,b)$ and $\udim Y'=(c,d)$ and
$X=X'_r$ and $Y=Y'_r$. It follows that $(a,b)=(c,d)$.

(1) $\Longrightarrow$ (2), (3), and (3) $\Longrightarrow$ (4) are straightforward.

(4) $\Longrightarrow$ (3)\quad
Let $M,M'\in\mathcal{C}$ be two indecomposable modules
with quasi-lengths $r\neq r'$, respectively. Similarly, let $N,N'\in\mathcal{D}$ be indecomposable modules
with quasi-lengths $s\neq s'$, respectively.
Assume $\udim M=\udim N$ and $\udim M'=\udim N'$.
Let $M=X_r$ with $\udim X=(a,b)$ and $N=Y_s$ with $\udim N=(c,d)$.
Since $\udim\tau^{i}M'=\udim \tau^i N'$, we may assume, without loss of generality,
$M'=X_{r'}$.  Thus $N'=Y'_{s'}$, where $Y'$ is a quasi-simple module with $\udim Y'=(c',d')$ and
$Y'=\tau^iY$ for some $i$.\\
\underline{Case  $r,s,r'$ odd and $s'$ even.}
In this case, we have
$$A_r(a,b)=A_s(c,d), \quad A_{r'}(a,b)=A_{s'}(c',d').$$

\begin{equation}
\left\{
\begin{array}{rcl}
A_ra &=& A_sc,\\
A_rb &=& A_sd,\\
A_{r'}a &=& A_{s'}d',\\
A_{r'}b &=& A_{s'}(nd'-c').
\end{array}
\right.
\end{equation}
It follows that $$\frac{d}{c}=\frac{nd'-c'}{d'}.$$
This is impossible by Lemma \ref{compare}.\\
\underline{Case  $r,r'$ odd and $s,s'$ even.}
In this case, we have
$$A_r(a,b)=A_s(d,nd-c), \quad A_{r'}(a,b)=A_{s'}(c',d').$$

\begin{equation}
\left\{
\begin{array}{rcl}
A_ra &=& A_sd,\\
A_rb &=& A_s(nd-c),\\
A_{r'}a &=& A_{s'}d',\\
A_{r'}b &=& A_{s'}(nd'-c').
\end{array}
\right.
\end{equation}
It follows that $$\frac{nd-c}{d}=\frac{nd'-c}{d'}.$$
Since $(nd,n^2d-nc)$ and $(nd', n^2d'-nc')$ are dimension vectors
of $Y_2$ and $Y'_2$, respectively,
the above equality hold only if $d=d'$ and $c=c'$, i.e. $Y_2=Y_2'$ by Lemma  \ref{compare}.
Therefore, $\frac{A_{r'}}{A_r}=\frac{A_{s'}}{A_s}$.
This contradicts Lemma \ref{inequa1}. 

All the other cases follow similarly. 

(3) $\Longrightarrow$ (1) Let $X\in\mathcal{C}$ be a quasi-simple module such that
$\udim X$ is minimal in $\mathcal{C}$.  Then by assumption there is some $Y\in\mathcal{D}$ such that
$\udim X=\udim Y$. Thus $Y$ is quasi-simple,
by the minimality of  $\udim Y$.
\end{proof}


\begin{thebibliography}{99}

\bibitem{ARS} M. Auslander, I. Reiten and S. O. Smal\o,
{\sl Representation Theory of Artin Algebras}. Cambridge studies in
advanced mathematics  {\bf 36} (Cambridge University Press, Cambridge,
1995).

\bibitem{C} B. Chen,
{\sl Regular modules and quasi-lengths over $3$-Kronecker quiver: using
Fibonacci numbers}. Arch. Math. {\bf 95}(2010), 105--113.

\bibitem{DF} P. Donovan and M. R. Freislich, {\sl The representation theory of finite graphs and associated algebras},
Carleton Math. Lecture Note {\bf 5}, Carleton University, Ottawa, 1973.  


\bibitem{K} V. G. Kac, Root systems, representations of quivers and
invariant theory.   Invariant theory, Montecatini (1982), Lecture
Notes in Mathematics, {\bf 996}, Springer Verlag, (1983), 74--108.

\bibitem{KR} H. Kraft and C. Riedemann, Geometry of representation of quivers. Representations of algebras (Durham, 1985), 109--145, London Math. Soc. Lecture Note Ser. {\bf 116}, Cambridge Univ. Press, Cambridge, 1986. 

\bibitem{FR} P. Fahr and C. M. Ringel,
A partition formula for Fibonacci numbers. {\sl J. Integer Seq}.  {\bf 11}
(2008), Article 08.1.4.

\bibitem{G} P. Gabriel, Unzerlegbare Darstellungen I, {\sl Manuscr. Math} {\bf 6}(1972), 71--103; correction, ibid. 6 (1972), 309. 

\bibitem{N} L. A. Nazarova, 
Representation of quivers of infinite type(Russian), 
{\sl Izv. Akad, Nauk SSSR Ser. Mat.} {\bf 37}(1973), 752--791. 

\bibitem{R} C. M. Ringel, 
{\sl Tame algebras and quadratic forms} Lecture Notes in Mathematics {\bf 1099}
Springer Verlag, 1984. 





\bibitem{Z} Y. Zhang,
The modules in any component of the AR-quiver of a wild hereditary
Artin algebra are uniquely determined by their composition factors.
{\sl Arch. Math}. (Basel) {\bf 53} (1989), 250--251.







\end{thebibliography}
\end{document}